\documentclass{amsart}
\usepackage{amsmath}
\usepackage{amssymb}
\usepackage{amsthm}

\DeclareMathOperator{\tr}{tr}

\DeclareMathOperator{\Vol}{Vol}
\DeclareMathOperator{\dvol}{dvol}

\DeclareMathOperator{\Ric}{Ric}

\DeclareMathOperator{\Conf}{Conf}

\DeclareMathOperator{\spansymbol}{span}

\newcommand{\cd}{\widetilde{d}}
\newcommand{\cg}{\widetilde{g}}

\newcommand{\cf}{\widetilde{f}}

\newcommand{\cu}{\widetilde{u}}
\newcommand{\cv}{\widetilde{v}}
\newcommand{\cw}{\widetilde{w}}
\newcommand{\cx}{\widetilde{x}}

\newcommand{\cB}{\widetilde{B}}

\newcommand{\cD}{\widetilde{D}}

\newcommand{\cI}{\widetilde{I}}

\newcommand{\cL}{\widetilde{L}}

\newcommand{\cX}{\widetilde{X}}

\newcommand{\ctau}{\widetilde{\tau}}
\newcommand{\cnabla}{\widetilde{\nabla}}
\newcommand{\cdelta}{\widetilde{\delta}}
\newcommand{\cDelta}{\widetilde{\Delta}}
\newcommand{\cmE}{\widetilde{\mathcal{E}}}

\newcommand{\cmL}{\widetilde{\mathcal{L}}}
\newcommand{\cmS}{\widetilde{\mathcal{S}}}

\newcommand{\lp}{\langle}
\newcommand{\rp}{\rangle}
\newcommand{\lv}{\lvert}
\newcommand{\rv}{\rvert}
\newcommand{\lV}{\lVert}
\newcommand{\rV}{\rVert}

\newcommand{\mE}{\mathcal{E}}

\newcommand{\mL}{\mathcal{L}}

\newcommand{\mN}{\mathcal{N}}

\newcommand{\mS}{\mathcal{S}}

\newcommand{\mV}{\mathcal{V}}

\newcommand{\bN}{\mathbb{N}}

\newcommand{\bR}{\mathbb{R}}

\def\sideremark#1{\ifvmode\leavevmode\fi\vadjust{\vbox to0pt{\vss
 \hbox to 0pt{\hskip\hsize\hskip1em
 \vbox{\hsize3cm\tiny\raggedright\pretolerance10000
 \noindent #1\hfill}\hss}\vbox to8pt{\vfil}\vss}}}

\newcommand{\suchthat}{\mathrel{}\middle|\mathrel{}}

\newcommand{\comment}[1]{}

\newtheorem{thm}{Theorem}[section]
\newtheorem{prop}[thm]{Proposition}
\newtheorem{lem}[thm]{Lemma}
\newtheorem{cor}[thm]{Corollary}

\theoremstyle{definition}

\theoremstyle{remark}
\newtheorem{remark}[thm]{Remark}

\numberwithin{equation}{section}

\usepackage{mathtools}

\begin{document}

\title{The Frank--Lieb approach to sharp Sobolev inequalities}
\author{Jeffrey S. Case}
\thanks{JSC was supported by a grant from the Simons Foundation (Grant No.\ 524601)}
\address{109 McAllister Building \\ Penn State University \\ University Park, PA 16802 \\ USA}
\email{jscase@psu.edu}
\keywords{sharp constants, Sobolev inequality}
\subjclass[2010]{Primary 53A30; Secondary 46E35}
\begin{abstract}
 Frank and Lieb gave a new, rearrangement-free, proof of the sharp Hardy--Littlewood--Sobolev inequalities by exploiting their conformal covariance.  Using this they gave new proofs of sharp Sobolev inequalities for the embeddings $W^{k,2}(\bR^n)\hookrightarrow L^{\frac{2n}{n-2k}}(\bR^n)$.  We show that their argument gives a direct proof of the latter inequalities without passing through Hardy--Littlewood--Sobolev inequalities, and, moreover, a new proof of a sharp fully nonlinear Sobolev inequality involving the $\sigma_2$-curvature.  Our argument relies on nice commutator identities deduced using the Fefferman--Graham ambient metric.
\end{abstract}
\maketitle

\section{Introduction}
\label{sec:intro}

The higher-order sharp Sobolev inequality states that on the round $n$-sphere, for any positive integer $k<\frac{n}{2}$ and any $u\in C^\infty(S^n)$, it holds that
\begin{equation}
 \label{eqn:sharp_sobolev}
 \mE_{2k}(u) := \int_{S^n} u\,L_{2k}u \geq \frac{\Gamma\bigl(\frac{n+2k}{2}\bigr)}{\Gamma\bigl(\frac{n-2k}{2}\bigr)}\omega_n^{2k/n}\left(\int_{S^n} \lv u\rv^{\frac{2n}{n-2k}}\right)^{\frac{n-2k}{n}}
\end{equation}
with equality if and only if there is a constant $a\in\bR$ and a point $\xi\in B^{n+1}$, the unit ball in $\bR^{n+1}$, such that
\begin{equation}
 \label{eqn:standard_bubble}
 u(\zeta) = a\left(1+\xi\cdot\zeta\right)^{-\frac{n-2k}{2}} .
\end{equation}
Here $\omega_n$ is the volume of $S^n=\partial B^{n+1}$ and $L_{2k}$ is the GJMS operator~\cite{GJMS1992}
\begin{equation}
 \label{eqn:gjms}
 L_{2k} := \prod_{j=1}^k \left(-\Delta + \frac{(n-2j)(n+2j-2)}{4}\right) .
\end{equation}
The significance of the GJMS operators is that they are conformally covariant, and in particular~\eqref{eqn:sharp_sobolev} immediately implies a sharp functional inequality for the embedding $W^{k,2}(\bR^n)\subset L^{\frac{2n}{n-2k}}(\bR^n)$ when $n>2k$.

Beckner~\cite{Beckner1993} first proved~\eqref{eqn:sharp_sobolev} by noting that it is equivalent to a sharp Hardy--Littlewood--Sobolev inequality originally proven by Lieb~\cite{Lieb1983} using the Riesz rearrangement inequality.  The reliance of this proof on symmetric rearrangements means that it is not applicable to other settings, such as CR geometry, where there exist differential operators (e.g.\ \cite{GoverGraham2005,Graham1984}) but one cannot use symmetric rearrangements.  An alternative, rearrangement-free, proof of the sharp Hardy--Littlewood--Sobolev inequalities was later given by Frank and Lieb~\cite{FrankLieb2012b}.  This gives a new proof of~\eqref{eqn:sharp_sobolev} and leads to a proof of analogous inequalities on the CR spheres and Heisenberg group~\cite{FrankLieb2012a}.

There are other sharp Sobolev inequalities whose extremal functions take the form~\eqref{eqn:standard_bubble}.  For example, let $\sigma_2^g=\frac{1}{2(n-2)^2}\bigl(\frac{n}{4(n-1)}R^2-\lv\Ric\rv^2\bigr)$ denote the $\sigma_2$-curvature of a Riemannian metric~\cite{Viaclovsky2000}.  It is known that any conformally flat metric $g$ of positive scalar curvature on $S^n$ satisfies
\begin{equation}
 \label{eqn:sigma2_sobolev}
 \int_{S^n} \sigma_2^g\,\dvol_g \geq \frac{n(n-1)}{8}\omega_n^{4/n} \Vol_g(S^n)^{\frac{n-4}{n}}
\end{equation}
with equality if and only if $g$ has constant sectional curvature~\cite{GeWang2013,GuanWang2004}.  We can rewrite this to resemble~\eqref{eqn:sharp_sobolev} as follows: Let $L_{\sigma_2}\colon \bigl(C^\infty(S^n)\bigr)^3\to C^\infty(S^n)$ be the operator
\begin{multline*}
 L_{\sigma_2}(u,u,u) := \frac{1}{2}\delta\left(\lv\nabla u\rv^2\,du\right) - \frac{n-4}{16}\left( u\Delta\lv\nabla u\rv^2 - \delta\left((\Delta u^2)\,du\right)\right) \\ - \frac{n-1}{4}\left(\frac{n-4}{4}\right)^2u\Delta u^2 + \frac{n(n-1)}{8}\left(\frac{n-4}{4}\right)^3u^3
\end{multline*}
defined with respect to a metric $g_0$ of constant sectional curvature one on $S^n$.  Then~\eqref{eqn:sigma2_sobolev} is equivalently written
\begin{multline}
 \label{eqn:sigma2_sobolev_functional}
 \mE_{\sigma_2}(u,u,u) := \int_{S^n} u\,L_{\sigma_2}(u,u,u)\,\dvol_{g_0}  \\ \geq \frac{n(n-1)}{8}\left(\frac{n-4}{4}\right)^3\omega_n^{4/n}\left(\int_{S^n} u^{\frac{8}{n-4}}\dvol_{g_0}\right)^{\frac{n-4}{4}}
\end{multline}
for all positive $u\in C^\infty(S^n)$ such that $u^{\frac{8}{n-4}}g_0$ has positive scalar curvature.  See Remark~\ref{rk:sigma2-functional} for further discussion.

There are two proofs of~\eqref{eqn:sigma2_sobolev} in the literature, both of which classify the critical points of the functional $g\mapsto \Vol_g(S^n)^{-\frac{n-4}{n}}\int\sigma_2^g\,\dvol_g$ in the conformal class of the round metric.  Viaclovsky~\cite{Viaclovsky2000} gave a proof modeled on Obata's argument~\cite{Obata1971} for characterizing conformally flat metrics of constant scalar curvature on the sphere; see also~\cite{ChangGurskyYang2003b}.  A.\ Li and Y.\ Li~\cite{LiLi2003} gave a proof using the method of moving planes.  It is possible, but relatively difficult, to extend Obata's argument to classify locally spherical contact forms of constant scalar curvature on the sphere~\cite{JerisonLee1988,Wang2013}.  However, the method of moving planes is problematic because not all reflections across planes are CR isomorphisms.

The main results of this note are new proofs, following the Frank--Lieb argument, of the following two results characterizing local minimizers of the quotients implicit in~\eqref{eqn:sharp_sobolev} and~\eqref{eqn:sigma2_sobolev}.  We expect this argument can be further developed to classify local minimizers of other Sobolev quotients in Euclidean space and in the Heisenberg group.  Note that this argument has also been used to prove some sharp Sobolev trace inequalities~\cite{CaseWang2019}.

\begin{thm}
 \label{thm:classification}
 Let $(S^n,g_0)$ be the round $n$-sphere of constant sectional curvature one and let $u\in\mV_{2k}$,
 \begin{equation}
  \label{eqn:mC}
  \mV_{2k} := \left\{ u\in C^\infty(S^n) \suchthat \int_{S^n} \lv u\rv^{\frac{2n}{n-2k}}\,\dvol_{g_0} = \omega_n \right\} ,
 \end{equation}
 be a positive local minimizer of $\mE_{2k}\colon\mV_{2k}\to\bR$.  Then there is a $\xi\in B^{n+1}$ such that
 \[ u(\zeta) = \left(\frac{1+\xi\cdot\zeta}{(1-\lv\xi\rv^2)^{1/2}}\right)^{-\frac{n-2k}{2}} . \]
\end{thm}

\begin{thm}
 \label{thm:sigma2_classification}
 Let $(S^n,g_0)$ be the round $n$-sphere of constant sectional curvature one and let $u\in\mV_{1,4}$,
 \begin{equation}
  \label{eqn:mV4}
  \mV_{1,4} := \left\{ u\in \mV_4 \suchthat u>0, R^{u^{\frac{8}{n-4}}g} >0 \right\} ,
 \end{equation}
 be a local minimizer of $\mE_{\sigma_2}\colon\mV_{1,4}\to\bR$.  Then there is a $\xi\in B^{n+1}$ such that
 \[ u(\zeta) = \left(\frac{1+\xi\cdot\zeta}{(1-\lv\xi\rv^2)^{1/2}}\right)^{-\frac{n-2k}{2}} . \]
\end{thm}

The Frank--Lieb argument involves three elements.  First, the assumption of a local minimizer implies, via the second variation, a nice spectral estimate; see Proposition~\ref{prop:stable} for a precise statement in terms of a general conformally covariant functional.  Second, conformal covariance implies that one can assume that $u$ is ``balanced'' (see Lemma~\ref{lem:balanced}), and in particular use first spherical harmonics as test functions in the previous spectral estimate.  Third, one deduces that a balanced positive local minimizer is constant by computing a particular commutator formula involving the relevant operator and a first spherical harmonic.  The commutator formulae needed for Theorems~\ref{thm:classification} and~\ref{thm:sigma2_classification} are given by the following two results.

\begin{thm}
 \label{thm:gjms_commutator}
 Let $(S^n,g_0)$ be the round $n$-sphere of constant sectional curvature one, regarded as the unit ball in $\bR^{n+1}$, and let $L_{2k}$ denote the GJMS operator of order $2k$.  Then
 \[ \sum_{i=0}^n x^i[L_{2k},x^i] = k(n+2k-2)L_{2k-2} \]
 for all $u\in C^\infty(S^n)$, where $x^0,\dotsc,x^n$ are the standard coordinates on $\bR^{n+1}$ and
 \[ [L_{2k},x^i](u) := L_{2k}(x^iu) - x^iL_{2k}(u) . \]
\end{thm}

\begin{thm}
 \label{thm:sigma2_commutator}
 Let $(S^n,g_0)$ be the round $n$-sphere of constant sectional curvature one, regarded as the unit ball in $\bR^{n+1}$, and let $L_{\sigma_2}$ denote the $\sigma_2$-operator.  Then
 \[ \sum_{i=0}^n x^i[L_{\sigma_2},x^i](u,u,u) = \frac{n-1}{3}u\sigma_1(u) \]
 for all $u\in C^\infty(S^n)$, where $x^0,\dotsc,x^n$ are the standard coordinates on $\bR^{n+1}$ and
 \[ [L_{\sigma_2},x^i](u,u,u) := L_{\sigma_2}(x^iu,u,u) - x^iL_{\sigma_2}(u,u,u) . \]
\end{thm}

Here
\begin{equation}
 \label{eqn:sigma1}
 \sigma_1(u) := -\frac{n-4}{8}\Delta u^2 - \lv\nabla u\rv^2 + \frac{n}{2}\left(\frac{n-4}{4}\right)^2u^2 .
\end{equation}
This operator is distinguished by the fact that it is a bidifferential operator which, when $u$ is positive, computes, up to multiplication by a constant and a power of $u$, the scalar curvature of $u^{\frac{8}{n-4}}g_0$; see Remark~\ref{rk:sigma1} for a precise formula.

We give simple proofs of Theorems~\ref{thm:gjms_commutator} and~\ref{thm:sigma2_commutator} using the Fefferman--Graham ambient space~\cite{FeffermanGraham2012}, where the formulae for $L_{2k}$ and $L_{\sigma_2}$ simplify considerably.  Interestingly, this technique also gives a conformally covariant identity relating the commutator of $L_{2k}$ and a first spherical harmonic to $L_{2k-2}$; see Proposition~\ref{prop:commutator}. We expect this technique to apply to a wide variety of conformally covariant operators (cf.\ \cite{CaseLinYuan2018b}).  In the cases studied here, the commutators in Theorems~\ref{thm:gjms_commutator} and ~\ref{thm:sigma2_commutator} have good properties which, when combined with the aforementioned spectral estimate, yield the proofs of Theorems~\ref{thm:classification} and~\ref{thm:sigma2_classification}.

This note is organized as follows.  In Section~\ref{sec:ambient} we recall some necessary facts about Fefferman--Graham ambient spaces, including the ambient definition of the GJMS and $L_{\sigma_2}$-operators.  In Section~\ref{sec:commutator} we prove Theorems~\ref{thm:gjms_commutator} and~\ref{thm:sigma2_commutator}.  In Section~\ref{sec:hersch} we deduce spectral estimates for local minimizers of a general conformally covariant functional and then apply them to prove Theorems~\ref{thm:classification} and~\ref{thm:sigma2_classification}.
\section{Some facts about the ambient space}
\label{sec:ambient}

Let $(S^n,g_0)$ denote the round $n$-sphere of constant sectional curvature one and let $(\bR_+^{n+1,1},\cg_0)$ denote the $(n+2)$-dimensional upper Minkowski space with coordinates $(\cx^0,\dotsc,\cx^n,\ctau)\in\bR^{n+1}\times(0,\infty)$ and metric
\[ \cg_0 = -d\ctau\otimes d\ctau + \sum_{i=0}^n d\cx^i\otimes d\cx^i . \]
Let $T:=\cx^i\partial_i$ denote the Euler vector field on $\bR_+^{n+1,1}$ and let
\[ \mN := \left\{ (\cx,\ctau) \in \bR_+^{n+1,1} \suchthat \lv \cx\rv^2 = \ctau^2 \right\} \]
denote the positive null cone.

The projectivization $P\mN$ of the positive null cone is diffeomorphic to $S^n$.  Let $U\colon S^n\cong P\mN\to\mN$ be a section of $\mN$.  Then
\begin{equation}
 \label{eqn:section}
 U(\zeta)=\bigl(u(\zeta)\zeta,u(\zeta)\bigr) 
\end{equation}
for a positive function $u\in C^\infty(S^n)$.  As $T$ is null with respect to the restriction of $\cg_0$ to $\mN$, each section $U$ of $\mN$ determines a Riemannian metric on $S^n$ by considering the inner product induced by $\cg_0\rv_{\mN}$ on $T_{U(x)}\mN/\spansymbol T\cong T_{x} S^n$ for each $x\in S^n$.  This identifies the section~\eqref{eqn:section} with $(S^n,u^2g_0)$.  In particular, we identify
\[ (S^n,g_0) = \left\{ (\cx,\ctau)\in\mN \suchthat \ctau=1 \right\} . \]

Given $w\in\bR$, denote by
\[ \cmE[w] := \left\{ \cf\in C^\infty(\bR_+^{n+1,1}) \suchthat T\cf = w\cf \right\} \]
the space of homogeneous functions of degree $w$ in Minkowski upper half space.  The conformal density bundle is the restriction
\[ \mE[w] := \left\{ \cf\rv_{\mN} \suchthat \cf\in\cmE[w] \right\} \]
of $\cmE[w]$ to $\mN$.  Given  a positive function $u\in C^\infty(S^n)$, define $E_w^u\colon C^\infty(S^n)\to\mE[w]$ by
\[ E_w^u(v)(\cx,\ctau) := \left(\frac{\ctau}{u(\cx/\ctau)}\right)^wv\left(\frac{\cx}{\ctau}\right) . \]
Note that $E_w^u(v)\bigl(u(\zeta)\zeta,u(\zeta)\bigr)=v(\zeta)$; that is, $E_w^u$ gives the standard trivialization of $\mE[w]$ induced by a metric $u^2g_0$.

Denote $Q:=\lv\cx\rv^2-\ctau^2$.  Then $Q\in\cmE[2]$ is a defining function for $\mN\subset\bR_+^{n+1,1}$; that is, $\mN=Q^{-1}(0)$ and $dQ\rv_\mN$ is nowhere vanishing.  An operator $\cD\colon\cmE[w]\to\cmE[w^\prime]$ is \emph{tangential} if $\cD(\cf)\rv_{\mN}$ depends only on $\cf\rv_{\mN}$.  Since $Q$ is a defining function for $\mN$, we see that $\cD$ is tangential if and only if $\cD(Q\cf)\rv_\mN=0$ for all $\cf\in\cmE[w-2]$.  Importantly, if $\cD\colon\cmE[w]\to\cmE[w^\prime]$ is tangential, then it restricts to $\mN$ as a conformally covariant operator of bidegree $(-w,-w^\prime)$.  More concretely, let $u^2g_0$ be a conformally flat metric on $S^n$.  We define the \emph{associated operator} $D_{u^2g_0}\colon C^\infty(M)\to C^\infty(S^n)$ by
\[ D_{u^2g_0}(v)(\zeta) := \cD(E_w^uv)\bigl(u(\zeta)\zeta,u(\zeta)\bigr) \]
for all $v\in C^\infty(S^n)$.  Since $\cD$ is tangential, we conclude that $D_{u^2g_0}$ is well-defined.  Moreover, the homogeneity assumptions on $\cD$ imply that
\[ D_{e^{2\Upsilon}g} = e^{w^\prime\Upsilon} \circ D_g \circ e^{-w\Upsilon} \]
for all $\Upsilon\in C^\infty(S^n)$, where $e^{w^\prime\Upsilon}$ and $e^{-w\Upsilon}$ are regarded as multiplication operators.

In this note we are primarily interested in two families of conformally covariant operators, both of which come from tangential operators in the ambient space.

First, let $\cDelta$ denote the Laplacian of $(\bR_+^{n+1,1},\cg_0)$.  For any $k\in\bN$, the operator $(-\cDelta)^k\colon\cmE\bigl[-\frac{n-2k}{2}\bigr]\to\cmE\bigl[-\frac{n+2k}{2}\bigr]$ is tangential~\cite{GJMS1992}.  The associated operator $L_{2k}\colon\mE\bigl[-\frac{n-2k}{2}\bigr]\to\mE\bigl[-\frac{n-2k}{2}\bigr]$, when evaluated at $g_0$, is the GJMS operator~\eqref{eqn:gjms}.  The fact that $L_{2k}$ is conformally covariant of bidegree $\bigl(\frac{n-2k}{2},\frac{n+2k}{2}\bigr)$ implies that
\[ \mE_{2k}^{e^{2\Upsilon}g_0}(u) = \mE_{2k}^{g_0}\left(e^{\frac{n-2k}{2}\Upsilon}u\right) \]
for all $u,\Upsilon\in C^\infty(S^n)$, where
\[ \mE_{2k}^g(u) := \int_{S^n} u\,L_{2k}^gu\,\dvol_g . \]

Second, let $\cdelta=\tr_{\cg}\cnabla$ denote the divergence and $\cd$ denote the exterior derivative on $(\bR_+^{n+1,1},\cg_0)$.  Define $\cB\colon\bigl(\cmE\bigl[-\frac{n-4}{4}\bigr]\bigr)^3\to\cmE\bigl[-\frac{3n+4}{4}\bigr]$ by
\begin{equation}
 \label{eqn:cB}
 \cB(\cu,\cv,\cw) := \frac{1}{2}\cdelta\left(\lp\cnabla u,\cnabla v\rp\,\cd\cw\right) - \frac{n-4}{16}\left[ \cw\cDelta\lp\cnabla u,\cnabla v\rp - \cdelta\left(\cDelta(uv)\,\cd w\right) \right] .
\end{equation}
Then $\cB$ is tangential~\cite{CaseLinYuan2018b}.  Define
\begin{equation}
 \label{eqn:ambient-sigma2-operator}
 \cL_{\sigma_2}(\cu,\cv,\cw) := \frac{1}{3}\left(\cB(\cu,\cv,\cw) + \cB(\cw,\cu,\cv) + \cB(\cv,\cw,\cu)\right) ,
\end{equation}
so that $\cL_{\sigma_2}$ is tangential on $\bigl(\cmE\bigl[-\frac{n-4}{4}\bigr]\bigr)^3$ and symmetric in its arguments.  We require the formula for the operators $L_{\sigma_2}$ associated to $\cL_{\sigma_2}$.

\begin{lem}
 \label{lem:sigma2-operator}
 Let $(S^n,g)$ be a conformally flat manifold.  Then the operator $L_{\sigma_2}$ associated to $\cL_{\sigma_2}$ is the polarization of
 \begin{multline}
  \label{eqn:sigma2-operator}
  L_{\sigma_2}(u,u,u) = \frac{1}{2}\delta\left(\lv\nabla u\rv^2\,du\right) - \frac{n-4}{16}\left(u\Delta\lv\nabla u\rv^2 - \delta\left((\Delta u^2)\,du\right)\right) \\ - \frac{1}{2}\left(\frac{n-4}{4}\right)^2 u\delta\left(T_1(\nabla u^2)\right) + \left(\frac{n-4}{4}\right)^3\sigma_2u^3 ,
 \end{multline}
 where $\sigma_2:=\frac{1}{2}(J^2-\lv P\rv^2)$ is the $\sigma_2$-curvature and $T_1:=Jg-P$ is the first Newton tensor, stated in terms of the Schouten tensor $P:=\frac{1}{n-2}(\Ric-Jg)$ and its trace $J:=\tr_gP=\frac{R}{2(n-1)}$.
\end{lem}

\begin{proof}
 This follows easily from the expansion~\cite{FeffermanGraham2012}
 \[ \cg_0 = 2\rho\,dt\otimes dt + t(dt\otimes d\rho + d\rho\otimes dt) + t^2(g + 2\rho P + \rho^2P^2) \]
 of the ambient metric and the fact that $(S^n,g)=t^{-1}(1)\cap\rho^{-1}(0)$ in these coordinates.
\end{proof}

\begin{remark}
 In fact, both $\cB$ and $\cL_{\sigma_2}$ are tangential for arbitrary ambient spaces~\cite{CaseLinYuan2018b}.  The computation in the proof of Lemma~\ref{lem:sigma2-operator} requires only the fact that the ambient metric takes the form
 \[ \cg_0 = 2\rho\,dt^2 + 2t\,dt\,d\rho + t^2g_\rho \]
 for $g_\rho$ a one-parameter family of metrics on $M$ with $g_\rho=g+2\rho P+O(\rho^2)$ and $\tr_g g_\rho = n + 2\rho J + \rho^2\lv P\rv^2 + O(\rho^3)$.  These properties hold in all dimensions $n\geq3$~\cite{FeffermanGraham2012}.
\end{remark}

We also require the formula for the Dirichlet form associated to $L_{\sigma_2}^{g_0}$.

\begin{cor}
 \label{cor:sigma2-dirichlet-form}
 Let $(S^n,g_0)$ be the round $n$-sphere.  Then
 \begin{multline*}
  \int_{S^n} u\,L_{\sigma_2}(u,u,u) = \int_{S^n} \biggl[ \lv\nabla u\rv^2\sigma_1(u) + \frac{1}{2}\lv\nabla u\rv^4 \\ + \frac{n-2}{2}\left(\frac{n-4}{4}\right)^2u^2\lv\nabla u\rv^2 + \frac{n(n-1)}{8}\left(\frac{n-4}{4}\right)^3u^4 \biggr] ,
 \end{multline*}
 where $\sigma_1(u)$ is given by~\eqref{eqn:sigma1} and all geometric quantities are defined using $g_0$.
\end{cor}

\begin{remark}
 \label{rk:sigma1}
 Note that if $u>0$ and $g_u=u^{\frac{8}{n-4}}g_0$, then
 \[ \sigma_1(u) = \left(\frac{n-4}{4}\right)^2u^{\frac{2n}{n-4}}\sigma_1^{g_u} . \]
 In particular, when $n\not=4$, $\sigma_1(u)>0$ if and only if $g_u$ has positive scalar curvature. 
\end{remark}

\begin{proof}
 Recall that the Schouten tensor of the round $n$-sphere satisfies $P=\frac{1}{2}g_0$.  Hence $J=\frac{n}{2}$ and $T_1=\frac{n-1}{2}g_0$.  We conclude from Lemma~\ref{lem:sigma2-operator} that
 \begin{multline*}
  \int_{S^n} u\,L_{\sigma_2}(u,u,u)\,\dvol = \int \biggl[ -\frac{1}{2}\lv\nabla u\rv^4 - \frac{n-4}{8}\lv\nabla u\rv^2\Delta u^2\\ + (n-1)\left(\frac{n-4}{4}\right)^2u^2\lv\nabla u\rv^2 + \frac{n(n-1)}{8}\left(\frac{n-4}{4}\right)^3u^4 \biggr] \dvol .
 \end{multline*}
 Writing this in terms of~\eqref{eqn:sigma1} yields the desired conclusion.
\end{proof}

\begin{remark}
 \label{rk:sigma2-functional}
 Corollary~\ref{cor:sigma2-dirichlet-form} and the properties of the operator $L_{\sigma_2}$ yield additional information about the fully nonlinear sharp Sobolev inequality~\eqref{eqn:sigma2_sobolev_functional}.
 
 First, set
 \[ \cmS_{1,4} := \left\{ u\in C^\infty(S^n) \suchthat \sigma_1(u)>0 \right\} . \]
 It follows from~\eqref{eqn:sigma1} that if $\sigma_1(u)>0$, then $u$ is nowhere vanishing.  Since $\sigma_1(-u)=\sigma_1(u)$ and $\mE_{\sigma_2}(-u)=\mE_{\sigma_2}(u)$, there is no loss in restricting to
 \[ \mS_{1,4} = \left\{ u \in \cmS_{1,4} \suchthat u>0 \right\} . \]
 Specifically, the infima of $\mE_{\sigma_2}\colon\cmS_{1,4}\cap\mV_4\to\bR$ and $\mE_{\sigma_2}\colon\mS_{1,4}\cap\mV_4\to\bR$ agree, where $\mV_4$ is given by~\eqref{eqn:mC}.  Note that $\mV_{1,4}$, as defined by~\eqref{eqn:mV4}, equals $\mS_{1,4}\cap\mV_4$.
 
 Second, it follows immediately from Corollary~\ref{cor:sigma2-dirichlet-form} and the continuity of the Sobolev embedding $W^{1,4}(S^n)\hookrightarrow L^{\frac{4n}{n-4}}(S^n)$ that
 \[ \inf \left\{ \mE_{\sigma_2}(u) \suchthat u \in \mV_{1,4} \right\} > 0 \]
 (cf.\ \cite{GeWang2013}).  Smooth positive minimizers of this infimum exist~\cite{GeWang2013,GuanWang2003b}.  Therefore Theorem~\ref{thm:sigma2_classification} identifies the minimizers, recovering~\eqref{eqn:sigma2_sobolev_functional}.
 
 Finally, by conformal covariance, we may rewrite~\eqref{eqn:sigma2_sobolev_functional} on $\bR^n$.  Specifically, Lemma~\ref{lem:sigma2-operator} gives the formula for $L_{\sigma_2}$ on $\bR^n$, whereupon one deduces that
 \[ \int_{\bR^n} \left[ I(u)\lv\nabla u\rv^2 + \frac{1}{2}\lv\nabla u\rv^4 \right] \geq \frac{n(n-1)}{8}\left(\frac{n-4}{4}\right)^3\omega_n^{4/n}\left(\int_{\bR^n} \lv u\rv^{\frac{4n}{n-4}}\right)^{\frac{n-4}{n}} \]
 for all suitably integrable $u\in C^\infty(\bR^n)$ such that $I(u):=-\frac{n-4}{8}\Delta u^2-\lv\nabla u\rv^2>0$, with equality if and only if
 \[ u(x) = a\left(1+\lv x-x_0\rv^2\right)^{-\frac{n-4}{n}} \]
 for some $a\in\bR$ and some $x_0\in\bR^n$.
\end{remark}

We require one last tangential operator.  Define the vector-valued operator $\cD$ by
\begin{equation}
 \label{eqn:thomas-D}
 \cD(\cu) := \cnabla(n+2T-2)\cu - (\cDelta u)T ,
\end{equation}
where $(\cDelta u)T$ denotes the pointwise multiplication of the vector field $T$ by the function $(\cDelta u)$.  It is straightforward to check that $\cD\colon\cmE[w]\to\cmE[w-1]$ is tangential for all $w\in\bR$; indeed, $\cD$ is the ambient form of the tractor-$D$ operator~\cite{CapGover2003}.  The map $\cu\mapsto\lv\cD(\cu)\rv^2$ is therefore tangential.  This operator is closely related to the scalar curvature:

\begin{lem}
 \label{lem:thomas-D}
 Let $(S^n,g_0)$ be the round $n$-sphere and let $(\bR_+^{n+1,1},\cg_0)$ be its ambient space.  For any real $w\not=-\frac{n-2}{2}$, the operator $\cI\colon\cmE[w]\to\cmE[2w-2]$ given by
 \begin{equation}
  \label{eqn:thomas-D-length}
  \cI(\cu) := -\frac{1}{2(n+2w-2)}\lv\cD(\cu)\rv^2 = -\frac{n+2w-2}{2}\lv\cnabla u\rv^2 + wu\cDelta u
 \end{equation}
 is tangential.  Moreover, if $u\in\mE[w]$ is positive and $w\not=0,-\frac{n-2}{2}$, then
 \begin{equation}
  \label{eqn:I-to-scal}
  I(u) = w^2u^{\frac{2(w-1)}{w}} J^{g_u} ,
 \end{equation}
 where $g_u=u^{-2/w}g_0$ and $J=\frac{R}{2(n-1)}$ is a rescaling of the scalar curvature.
\end{lem}

\begin{proof}
 Equation~\eqref{eqn:thomas-D-length} follows immediately from the definition of $\cD$, while the fact that $\cI$ is tangential follows immediately from the fact that $\cD$ is tangential.  Since $\cD$ induces the tractor-$D$ operator~\cite{CapGover2003}, it holds that
 \begin{equation}
  \label{eqn:eval-I}
  I(u) = -\frac{n+2w-2}{2}\lv\nabla u\rv^2 + wu(\Delta + wJ)u
 \end{equation}
 (cf.\ \cite{Bailey1994}).  In particular, $I(1)=w^2J^g$.  Equation~\eqref{eqn:I-to-scal} then follows from conformal covariance.
\end{proof}

\begin{remark}
 Equation~\eqref{eqn:eval-I} implies that $I(u)=\sigma_1(u)$ when computed with respect to $g_0$ and that $I(u)$ is as in Remark~\ref{rk:sigma2-functional} when computed with respect to $dx^2$.  It is for these reasons we introduced the normalization factor in~\eqref{eqn:thomas-D-length}.
\end{remark}

\section{Commutator identities}
\label{sec:commutator}

In this section we prove some commutator identities involving the GJMS and $\sigma_2$-operators.  We prove two families of identities.  First, we deduce a conformally covariant commutator identity for $[L_{2k},x]$ on the round $n$-sphere which is stated in terms of $L_{2k-2}$ and a conformal Killing field, where $x$ is any first spherical harmonic.  Second, we deduce conformally covariant commutator identities for $\sum x^i[L,x^i]$ for both the GJMS and $\sigma_2$-operators, where $x^i$ are a basis of first spherical harmonics.  The former identity is interesting and stronger than the latter identity, but only the latter identity is necessary to execute the Frank--Lieb argument.

Let $(S^n,g_0)$ be the round $n$-sphere and let $(\bR_+^{n+1,1},\cg_0)$ be its ambient space as in Section~\ref{sec:ambient}.  For any $0\leq i\leq n$, consider the vector field
\begin{equation}
 \label{eqn:cX}
 \cX(i) := \ctau\cnabla\cx^i - \cx^i\cnabla\ctau .
\end{equation}
It is easily checked that $\cmL_{\cX(i)}\colon\cmE[w]\to\cmE[w]$ is tangential.  Note that the restriction of $\{\cX(i)\}_{i=0}^n$ to $(S^n,g_0)$ is a basis for its space of essential conformal Killing fields.  The associated operator is the conformally invariant operator
\begin{equation}
 \label{eqn:first-order-operator}
 \mL_X - \frac{w}{n}\delta X \colon \mE[w] \to \mE[w] ,
\end{equation}
where $X=X(i)$ and $\delta X$ is the divergence of $X$, regarded as a multiplication operator.  Specializing further to $(S^n,g_0)$, this is the operator
\begin{equation}
 \label{eqn:first-order-conformal-operator}
 \mL_{\nabla x^i} + wx^i \colon \mE[w] \to \mE[w] .
\end{equation}
Our first family of commutator identities will be expressed in terms of $\cmL_{\cX(i)}$.


Our second family of commutator identities will take advantage of a nice formula for the sums $\sum \cx^i\cX(i)$.

\begin{lem}
 \label{lem:sum_cX}
 Let $(\bR_+^{n+1,1},\cg_0)$ be Minkowski upper half space.  Then
 \[ \sum_{i=0}^n \cx^i\cX(i) = \ctau T + Q\partial_{\ctau} . \]
\end{lem}

\begin{proof}
 Note that, because of the signature of $\cg_0$,
 \[ \cX(i) = \ctau\partial_{\cx^i} + \cx^i\partial_{\ctau} . \]
 The conclusion readily follows from this and the definitions of $Q$ and $T$.
\end{proof}

We now turn to the desired commutator identities involving the GJMS operators.  We begin with a simple observation about commutators of the ambient Laplacian.

\begin{lem}
 \label{lem:commute}
 Let $\cx$ be such that $\cnabla^2\cx=0$ on Minkowski space $(\bR_+^{n+1,1},\cg_0)$.  Let $k\in\bN$ and $u\in C^\infty(\bR_+^{n+1,1})$.  Then
 \[ \bigl[(-\cDelta)^k,\cx\bigr] = -2k\cmL_{\cX}\circ(-\cDelta)^{k-1} . \]
 where $\cX:=\cnabla\cx$ is the gradient of $\cx$ with respect to $\cg_0$.
\end{lem}

\begin{proof}
 Since $\cnabla^2\cx=0$ and $\cg_0$ is Ricci flat, we have that
 \[ \bigl[-\cDelta,\cx\bigr] = -2\mL_{\cX}, \qquad \bigl[-\cDelta,\mL_{\cX}\bigr] = 0 . \]
 The conclusion readily follows.
\end{proof}

Lemma~\ref{lem:commute} immediately yields a formula for $\ctau\cDelta\cx^i-\cx^i\cDelta \ctau$.  Expressing this in terms of tangential operators yields our first commutator formula involving the GJMS operators.

\begin{prop}
 \label{prop:commutator}
 Let $(S^n,g_0)$ be the round $n$-sphere and let $(\bR_+^{n+1,1},\cg_0)$ be its ambient space.  Given integers $k\geq1$ and $0\leq i\leq n$, it holds that
 \begin{equation}
  \label{eqn:tangential_commutator}
  \ctau(-\cDelta)^k\cx^i - \cx^i(-\cDelta)^k\ctau = -2k\cmL_{\cX(i)}\circ(-\cDelta)^{k-1} .
 \end{equation}
 Moreover, all summands in~\eqref{eqn:tangential_commutator} are tangential on $\cmE\bigl[-\frac{n-2k+2}{2}\bigr]$.
\end{prop}

\begin{proof}
 Equation~\ref{eqn:tangential_commutator} follows immediately from Lemma~\ref{lem:commute}.  Since $\cx^i,\ctau\in\cmE[1]$, each of $(-\cDelta)^k\cx^i$, $(-\cDelta)^k\ctau$, and $(-\cDelta)^{k-1}$ are tangential on $\cmE\bigl[-\frac{n-2k+2}{2}\bigr]$.  The final conclusion follows from the fact that $\cmL_{\cX(i)}$ is tangential on $\cmE[w]$ for any $w\in\bR$.
\end{proof}

Though unnecessary for the rest of this note, the following corollary is expected to find more general applications.

\begin{cor}
 \label{cor:gjms_commute}
 Let $(S^n,g_0)$ be the round $n$-sphere and let $x$ be a first spherical harmonic.  Given $k\in\bN$, it holds that
 \[ [L_{2k},x] = k\bigl((n+2k-2)x-2\mL_X\bigr)L_{2k-2}, \]
 where $X=\nabla x$ is the gradient of $x$ with respect to $g_0$.
\end{cor}

\begin{proof}
 Restricting to $(S^n,g_0)$ implies that $\ctau=1$ and that $\cmL_{\cX(i)}$ is given by~\eqref{eqn:first-order-conformal-operator}.  Combining this observation with~\eqref{eqn:tangential_commutator} yields the desired result.
\end{proof}

We can now prove the commutator identity involving the GJMS operators needed to execute the Frank--Lieb argument:

\begin{proof}[Proof of Theorem~\ref{thm:gjms_commutator}]
 Combining Lemma~\ref{lem:sum_cX} and Proposition~\ref{prop:commutator} yields the operator identity
 \[ \sum_{i=0}^n \ctau\cx^i(-\cDelta)^k\cx^i \equiv \tau^2(-\cDelta)^k\ctau - 2k\ctau \mL_T\circ(-\cDelta)^{k-1} \mod Q . \]
 Restricting this to $\cmE\bigl[-\frac{n-2k+2}{2}\bigr]$ yields
 \[ \sum_{i=0}^n \ctau\cx^i(-\cDelta)^k\cx^i \equiv \tau^2(-\cDelta)^k\ctau + k(n+2k-2)(-\cDelta)^{k-1} \mod Q . \]
 Since all summands are tangential on $\cmE\bigl[-\frac{n-2k+2}{2}\bigr]$, we see that this descends to a conformally covariant operator identity on the conformal $n$-sphere.  Specializing to $(S^n,g_0)$, where $\ctau=1$, yields the final conclusion.
\end{proof}

The same techniques as used above also establish the commutator identity involving the operator $L_{\sigma_2}$ needed to execute the Frank--Lieb argument.

\begin{proof}[Proof of Theorem~\ref{thm:sigma2_commutator}]
 Let $\cB$ be as in~\eqref{eqn:cB}, let $\cu\in\cmE\bigl[-\frac{n-4}{4}\bigr]$, and let $\cv\in\cmE\bigl[-\frac{n}{4}\bigr]$.  Let $\cx$ be such that $\cnabla^2\cx=0$.  Then
 \begin{align*}
  \cB(\cu,\cu,\cx^i\cv) & = \cx^i\cB(\cu,\cu,\cv) - \lp\cnabla v,\cnabla x\rp\,\cI(\cu) - \frac{1}{2}\cv\lp\cnabla\cx^i,\cnabla\cI(\cu)\rp, \\
  \cB(\cx^i\cv,\cu,\cu) & = \cx^i\cB(\cv,\cu,\cu) + \frac{1}{4}\cv\lp\cnabla\cx^i,\cnabla\cI(\cu)\rp - \frac{1}{2}\lp\cnabla\cv,\cnabla\cx\rp\,\cI(\cu) \\
   & \quad - \frac{n-4}{4}\cu\cnabla^2\cu(\cnabla\cv,\cnabla\cx) - \frac{1}{2}\lv\cnabla\cu\rv^2\lp\cnabla\cv,\cnabla\cx\rp + \frac{n}{4}\lp\cnabla\cu,\cnabla\cx\rp\lp\cnabla\cu,\cnabla\cv\rp \\
   & \quad + \frac{n}{8}\cv\lp\cnabla\cx,\cnabla\lv\cnabla\cu\rv^2\rp + \frac{n-2}{4}\cv\lp\cnabla\cu,\cnabla\cx\rp\,\cDelta\cu ,
 \end{align*}
 where $\cI\colon\cmE\bigl[-\frac{n-4}{4}\bigr]\to\cmE\bigl[-\frac{n}{2}\bigr]$ is given by~\eqref{eqn:thomas-D-length}.  We conclude that
 \begin{align*}
  \MoveEqLeft[2] \ctau\cL_{\sigma_2}(\cx^i\cv,\cu,\cu) - \cx^i\cL_{\sigma_2}(\ctau\cv,\cu,\cu) \\
  & = -\frac{2}{3}\lp\cX(i),\cnabla\cv\rp\,\cI(\cu) - \frac{n-4}{6}\cu\cnabla^2\cu(\cX(i),\cnabla\cv) - \frac{1}{3}\lv\cnabla\cu\rv^2\lp\cX(i),\cnabla\cv\rp \\
   & \quad + \frac{n}{6}\lp\cX(i),\cnabla\cu\rp\lp\cnabla\cu,\cnabla\cv\rp + \frac{n}{12}\cv\lp\cX(i),\cnabla\lv\cnabla\cu\rv^2\rp + \frac{n-2}{6}\cv\lp\cX(i),\cnabla\cu\rp\,\cDelta\cu \\
 \end{align*}
 for all integers $0\leq i\leq n$.  Combining this with Lemma~\ref{lem:sum_cX} and the homogeneity assumptions on $\cu$ and $\cv$ yields
 \begin{equation}
  \label{eqn:ambient-sigma2-sum-commutator}
  \sum_{i=0}^n \cx^i\left[ \ctau\cL_{\sigma_2}(\cx^i\cv,\cu,\cu) - \cx^i\cL_{\sigma_2}(\ctau\cv,\cu,\cu)\right] \equiv \frac{n-1}{3}\ctau\cv\cI(\cu) \mod Q .
 \end{equation}
 Note that~\eqref{eqn:ambient-sigma2-sum-commutator} is an equivalence of tangential operators.  Evaluating~\eqref{eqn:ambient-sigma2-sum-commutator} at $(S^n,g_0)$ and using~\eqref{eqn:eval-I} yields the desired conclusion.
\end{proof}

\section{The Frank--Lieb argument}
\label{sec:hersch}

We begin by explaining how local minimizers of a conformally invariant functional give rise to spectral estimates.  Let $L$ be a formally self-adjoint conformally covariant $j$-differential operator of total order $2k$ on $S^n$, $n>2k$.  That is, each metric $g$ on $S^n$ determines an operator $L^g\colon\bigl(C^\infty(S^n)\bigr)^j\to C^\infty(S^n)$ that is multilinear over $\bR$, acts as a differential operator on each of its arguments, satisfies
\[ \phi^\ast\left(L^g(u_1,\dotsc,u_j)\right) = L^{\phi^\ast g}(\phi^\ast u_1,\dotsc,\phi^\ast u_j) \]
for each diffeomorphism $\phi$ of $S^n$, is such that
\[ (u_0,\dotsc,u_j) \mapsto \int_{S^n} u_0\,L^g(u_1,\dotsc,u_j)\,\dvol_g \]
is symmetric in $(u_0,\dotsc,u_j)\in\bigl(C^\infty(S^n)\bigr)^{j+1}$, and satisfies
\[ L^{g_w}(u_1,\dotsc,u_j) = w^{-\frac{jn+2k}{n-2k}}L^g\left(wu_1,\dotsc,wu_j\right) \]
for all positive functions $w\in C^\infty(S^n)$ and all functions $u_1,\dotsc,u_j\in C^\infty(S^n)$, where $g_w=w^{\frac{2(j+1)}{n-2k}}g$.  Note that the GJMS operators $L_{2k}$ are formally self-adjoint conformally covariant $1$-differential operators of total order $2k$ and the operator $L_{\sigma_2}$ is a formally self-adjoint conformally covariant $3$-differential operator of total order $4$.  See~\cite{CaseLinYuan2018b} for other constructions of such operators.

Let $L$ be as in the previous paragraph and define the \emph{associated Dirichlet energy} $\mE_L^g\colon C^\infty(S^n)\to\bR$ by
\[ \mE_L^g(u) := \int_{S^n} u\,L^g(u,\dotsc,u)\,\dvol_g . \]
The fact that $L$ is formally self-adjoint implies that if $u_t$ is a one-parameter family of functions with $u_0=u$, then
\begin{align}
 \label{eqn:first_variation} \left.\frac{d}{dt}\right|_{t=0}\mE_L^g(u_t) & = (j+1)\int_{S^n} \dot u\,L^g(u,\dotsc,u)\,\dvol_g, \\
 \label{eqn:second_variation} \left.\frac{d^2}{dt^2}\right|_{t=0}\mE_L^g(u_t) & = j(j+1)\int_{S^n} \dot u\,L^g(\dot u,u,\dotsc,u)\,\dvol_g \\
  \notag & \qquad + (j+1)\int_{S^n} \ddot u\,L^g(u,\dotsc,u)\,\dvol_g
\end{align}
for $\dot u:=\frac{\partial}{\partial t}\bigr|_{t=0}u_t$ and $\ddot u:=\frac{\partial^2}{\partial t^2}\bigr|_{t=0}u_t$, while the fact that $L$ is conformally covariant implies that
\[ \mE_L^{g_w}(u) = \mE_L^g(wu) \]
for all $w\in C^\infty(S^n)$ positive and all $u\in C^\infty(S^n)$, where $g_w=w^{\frac{2(j+1)}{n-2k}}g$.  Define
\[ \mV_{2k}^g := \left\{ u\in C^\infty(S^n) \suchthat \int_{S^n} \lv u\rv^{\frac{n(j+1)}{n-2k}}\,\dvol_g = \omega_n \right\} \]
and note that $u\in\mV_{2k}^{g_w}$ if and only if $wu\in\mV_{2k}^g$.  Thus the problem of finding critical points of $\mE^g\colon\mV_{2k}^g\to\bR$ is conformally invariant.

The first steps in the Frank--Lieb argument are to obtain a spectral estimate from the assumption that $u$ is a local positive minimizer of $\mE_L^g\colon\mV_{2k}^g\to\bR$ and apply the estimate specifically to the first spherical harmonics.  The result, under general assumptions, is as follows:

\begin{prop}
 \label{prop:stable}
 Let $L$ be a formally self-adjoint $j$-differential operator of total order $2k$ on $(S^n,g_0)$, $n>2k$, and suppose that $u$ is a positive local minimizer of $\mE_L^{g_0}\colon\mV_{2k}^{g_0}\to\bR$.  Suppose additionally that
 \begin{equation}
  \label{eqn:balanced}
  \int_{S^n} x^iu^{\frac{n(j+1)}{n-2k}}\,\dvol_{g_0} = 0
 \end{equation}
 for all integers $0\leq i\leq n$, where $x^0,\dotsc,x^n$ are coordinates on $\bR^{n+1}$ and we regard $S^n\subset\bR^{n+1}$ as the unit sphere.  Then
 \begin{equation}
  \label{eqn:stability_commutator}
  \sum_{i=0}^n \int_{S^n} x^iu\,[L^{g_0},x^i](u,\dotsc,u)\,\dvol_{g_0} \geq \frac{2(j+1)k}{j(n-2k)}\mE_L^{g_0}(u) ,
 \end{equation}
 where
 \[ [L^{g_0},x^i](u) := L^{g_0}(x^iu,u,\dotsc,u) - x^iL^{g_0}(u,u,\dotsc,u) . \]
\end{prop}

\begin{proof}
 Let $v\in C^\infty(S^n)$ be such that
 \begin{equation}
  \label{eqn:tangent}
  \int_{S^n} vu^{\frac{nj+2k}{n-2k}}\,\dvol_{g_0} = 0 .
 \end{equation}
 Then $u_t:=\omega_n^{\frac{n-2k}{n(j+1)}}\lV u+tv\rV_{\frac{n(j+1)}{n-2k}}^{-1}(u+tv)$ defines a smooth curve in $\mV_{2k}^{g_0}$ with $u_0=u$ and $\left.\frac{\partial}{\partial t}\right|_{t=0}u_t=v$.  Since $u$ is a critical point of $\mE_L\colon\mV_{2k}^{g_0}\to\bR$, it follows from~\eqref{eqn:first_variation} that
 \[ L^{g_0}(u,\dotsc,u) = \omega_n^{-1}\mE_L^{g_0}(u)u^{\frac{nj+2k}{n-2k}} . \]
 Since $u$ is a local minimizer, $\frac{d^2}{dt^2}\bigr|_{t=0}\mE_L^{g_0}(u_t)\geq0$.  Expanding this using~\eqref{eqn:second_variation} and the above display yields
 \begin{equation}
  \label{eqn:stability}
  \int_{S^n} v\,L^{g_0}(v,u,\dotsc,u)\,\dvol_{g_0} \geq \frac{nj+2k}{j(n-2k)}\frac{\mE_L^{g_0}(u)}{\omega_n}\int_{S^n} v^2u^{\frac{n(j-1)+4k}{n-2k}}\,\dvol_{g_0} .
 \end{equation}
 
 The assumption~\eqref{eqn:balanced} implies that for each integer $0\leq i\leq n$, the function $v=x^iu$ satisfies~\eqref{eqn:tangent}.  It follows from~\eqref{eqn:stability} that
 \[ \sum_{i=0}^n \int_{S^n} x^iu\,L^{g_0}(x^iu,u,\dotsc,u)\,\dvol_{g_0} \geq \frac{nj+2k}{j(n-2k)}\mE_L^{g_0}(u) . \]
 The final conclusion follows from the definition of the commutator $[L,x^i]$.
\end{proof}

The balancing assumption~\eqref{eqn:balanced} is used only to obtain the simple form~\eqref{eqn:stability_commutator}.  If $L$ is conformally covariant, one can always assume that a local minimizer is balanced:

\begin{lem}
 \label{lem:balanced}
 Let $L$ be a formally self-adjoint conformally covariant $j$-differential operator of total order $2k$ on $(S^n,g_0)$, $n>2k$, and suppose that $u\in C^\infty(S^n)$ is a positive critical point of $\mE^{g_0}\colon\mV_{2k}^{g_0}\to\bR$.  Then there is an element $\Phi$ of the conformal group $\Conf(S^n)$ of $S^n$ such that
 \[ u_\Phi := \lv J_\Phi\rv^{\frac{n-2k}{n(j+1)}}\Phi^\ast u \]
 is a critical point of $\mE^{g_0}\colon\mV_{2k}^{g_0}\to\bR$ which satisfies~\eqref{eqn:balanced}, where $\lv J_\Phi\rv$ is the determinant of the Jacobian of $\Phi$.
\end{lem}

\begin{proof}
 Since $\mE$ and $\mV_{2k}$ are conformally covariant, $u$ is a critical point of $\mE^{g_0}\colon\mV_{2k}^{g_0}\to\bR$ if and only if $u_\Phi$ is a critical point of $\mE^{g_0}\colon\mV_{2k}^{g_0}\to\bR$ for each $\Phi\in\Conf(S^n)$.  Since $u$ is positive, $\int u^{\frac{n(j+1)}{n-2k}}\not=0$.  It follows from~\cite[Lemma~B.1]{FrankLieb2012b} that there is a $\Phi\in\Conf(S^n)$ such that $u_\Phi$ satisfies~\eqref{eqn:balanced}.
\end{proof}

Proposition~\ref{prop:stable} and Lemma~\ref{lem:balanced} reduce the problem of classifying positive local minimizers of $\mE\colon\mV_{2k}\to\bR$ to the problem of showing that the only functions which satisfy~\eqref{eqn:stability_commutator} are the constants.  This can be done for the operators $L_{2k}$ and $L_{\sigma_2}$ using the results of Section~\ref{sec:commutator}.

\begin{proof}[Proof of Theorem~\ref{thm:classification}]
 All computations in this proof are carried out with respect to the round $n$-sphere of constant sectional curvature one.  Let $u$ be a positive local minimizer of
 \[ \mE_{2k}(v) := \int_{S^n} v\,L_{2k}v\,\dvol \]
 in $\mV_{2k}$.  By Lemma~\ref{lem:balanced}, we may assume that $u$ satisfies~\eqref{eqn:balanced}.  We conclude from Proposition~\ref{prop:stable} that
 \[ \sum_{i=0}^n \int_{S^n} x^iu\,[L_{2k},x^i](u)\,\dvol \geq \frac{4k}{n-2k}\mE(u) . \]
 Combining this with Theorem~\ref{thm:gjms_commutator} yields
 \begin{equation}
  \label{eqn:gjms_stable}
  0 \geq \frac{4k}{n-2k}\int_{S^n} u\left(L_{2k} - \frac{(n-2k)(n+2k-2)}{4}L_{2k-2}\right)u\,\dvol .
 \end{equation}
 It follows from the factorization~\eqref{eqn:gjms} that
 \[ L_{2k} - \frac{(n-2k)(n+2k-2)}{4}L_{2k-2} = -\Delta L_{2k-2} \]
 is a nonnegative operator with kernel exactly equal to the constant functions.  Combining this with~\eqref{eqn:gjms_stable} yields $u=1$.  The final conclusion follows from the fact that if $\Phi\in\Conf(S^n)$, then
 \begin{equation}
  \label{eqn:jacobian}
  \lv J_\Phi\rv(\zeta) = \left(\frac{1+\xi\cdot\zeta}{(1-\lv\xi\rv^2)^{1/2}}\right)^{-n}
 \end{equation}
 for some $\xi\in B^{n+1}$, the unit ball in $\bR^{n+1}$.
\end{proof}

\begin{proof}[Proof of Theorem~\ref{thm:sigma2_classification}]
 All computations in this proof are carried out with respect to the round $n$-sphere of constant sectional curvature one.  Note that $\mV_{1,4}$ is an open subset of $\mV_4$.  Let $u$ be a positive local minimizer of $\mE_{\sigma_2}\colon\mV_{1,4}\to\bR$.  By Lemma~\ref{lem:balanced}, we may assume that $u$ satisfies~\eqref{eqn:balanced}.  Proposition~\ref{prop:stable} implies that
 \[ \sum_{i=0}^n \int_{S^n} x^iu\,[L_{\sigma_2},x^i](u,u,u)\,\dvol \geq \frac{16}{3(n-4)}\mE_{\sigma_2}(u) . \]
 Combining this with Theorem~\ref{thm:sigma2_commutator} yields
 \[ 0 \geq \int_{S^n} \left[ \frac{16}{3(n-4)}\left(\sigma_1(u)+\frac{1}{2}\lv\nabla u\rv^2\right)\lv\nabla u\rv^2 + \frac{n+2}{6}u^2\lv\nabla u\rv^2\right]\,\dvol . \]
 Since $\sigma_1(u)>0$, we conclude that $u=1$.  The final conclusion follows from the fact that if $\Phi\in\Conf(S^n)$, then~\eqref{eqn:jacobian} holds for some $\xi\in B^{n+1}$.
\end{proof}


\bibliographystyle{abbrv}
\bibliography{../bib}
\end{document}